\documentclass[11pt]{amsart}
\usepackage{a4wide}
\usepackage{amsmath}
\usepackage[utf8]{inputenc}
\usepackage{amssymb}
\usepackage{amsopn}
\usepackage{epsfig}
\usepackage{amsfonts}
\usepackage{latexsym}
\usepackage{graphicx}
\usepackage{enumerate}

\newtheorem{theorem}{Theorem}[section]
\newtheorem{lemma}[theorem]{Lemma}

 \newtheorem{main}{Theorem}

\theoremstyle{definition}

\newtheorem{example}[theorem]{Example}

\theoremstyle{remark}
\newtheorem{remark}[theorem]{Remark}

\numberwithin{equation}{section}

\newcommand{\N}{\ensuremath{\mathbb{N}}}

\newcommand{\B}{\ensuremath{\mathcal{B}}}

\newcommand{\U}{{  {\mathcal{U}}}}

\newcommand{\set}[1]{\left\{#1\right\}}
\newcommand{\La}{\Lambda}

\newcommand{\f}{\infty}

\newcommand{\al}{\alpha}

\newcommand{\ra}{\rightarrow}
\newcommand{\lle}{\preccurlyeq}
\newcommand{\lge}{\succcurlyeq}

\begin{document}

\title{Multiple expansions of real numbers  with  digits set $\set{0,1,q}$}

\author[K.Dajani]{Karma Dajani}
\address[K. Dajani]{Department of Mathematics, Utrecht University,   Budapestlaan 6, P.O. Box 80.000, 3508 TA Utrecht, The Netherlands}
\email{k.dajani1@uu.nl}

\author[K.Jiang]{Kan Jiang}
\address[K. Jiang]{Department of Mathematics, Utrecht University,   Budapestlaan 6, P.O. Box 80.000, 3508 TA Utrecht, The Netherlands}
\curraddr{Department of Mathematics, Ningbo University, Ningbo, Zhejiang, People's Republic of China}
\email{kanjiangbunnik@yahoo.com}

\author[D.Kong]{Derong Kong}
\address[D. Kong]{Mathematical Institute, University of Leiden, PO Box 9512, 2300 RA Leiden, The Netherlands}
\curraddr{College of Mathematics and Statistics, Chongqing University, 401331, Chongqing, P.R.China}
\email{derongkong@126.com}

\author[W.Li]{Wenxia Li}
\address[W. Li]{School of Mathematical Sciences, Shanghai Key Laboratory of PMMP, East China Normal University, Shanghai 200062,
People's Republic of China}
\email{wxli@math.ecnu.edu.cn}


\dedicatory{ Dedicated to Michel Dekking on the occasion of his 70$^{th}$ birthday}



\begin{abstract}
For $q>1$  we consider expansions in base $q$ with digits set $\set{0,1,q}$. Let $\U_q$ be the set of points which have a unique $q$-expansion. For   $k=2, 3,\cdots,\aleph_0$ let $\B_k$ be the set of bases $q>1$ for which there exists $x$ having precisely $k$ different  $q$-expansions, and  for $q\in \B_k$ let $\U_q^{(k)}$ be the set of all such  $x$'s which have exactly $k$ different $q$-expansions. In this paper we show that 
\[
\B_{\aleph_0}=[2,\f)\quad\textrm{and}\quad
\B_k=(q_c,\f)\quad \textrm{for any}\quad k\ge 2,\]
 where 
$q_c\approx 2.32472$ is the appropriate root of
$x^3-3x^2+2x-1=0$.
 Moreover,  we show that for any  integer $k\ge 2$ and any  $q\in\B_{k}$ the Hausdorff dimensions of $\U_q^{(k)}$  and $\U_q$ are   the same, i.e.,
\[
\dim_H\U_q^{(k)}=\dim_H\U_q\quad\textrm{for any}\quad k\ge 2.
\]
Finally, we conclude that the set of points having a continuum of $q$-expansions has full Hausdorff dimension.
\end{abstract}

\maketitle

\section{Introduction}\label{sec: Introduction}

Expansions in non-integer bases were pioneered by R\'{e}nyi \cite{Renyi_1957} and Parry \cite{Parry_1960}. Unlike integer base expansions, for a given $\beta\in(1, 2)$, it is well-known  that typically a real number $x\in I_\beta:=[0, 1/(\beta-1)]$ has a  continuum of $\beta$-expansions with digits set $\set{0, 1}$ (cf.~\cite{Dajani_DeVries_2007, Sidorov_2003}), i.e., for Lebesuge almost every $x\in I_\beta$ there exist a continuum of zero-one sequences $(x_i)$ such that $x=\sum_{i=1}^\f x_i/\beta^i$.   However, there still exist $x\in I_\beta$ having a unique $\beta$-expansion  (cf.~\cite{Erdos_Joo_Komornik_1990, Glendinning_Sidorov_2001, Kong_Li_Dekking_2010}). Denote by $\U_\beta$ the set of all  $x\in I_\beta$ with a unique $\beta$-expansion.  De Vries and Komornik \cite{DeVries_Komornik_2008}
investigated the topological properties of $\U_\beta$. Komornik et al. \cite{Komornik_Kong_Li_2015_1} considered the Hausdorff dimension of $\U_\beta$, and concluded that the dimension function $\beta\mapsto \dim_H\U_\beta$ behaves like a Devil's staircase. Interestingly, for any $k=2,3,\cdots$ or $\aleph_0$   Erd\H{o}s et al. \cite{Erdos_Horvath_Joo_1991, Erdos_Joo_1992} showed that there exist  $\beta\in(1,2)$ and  $x\in I_\beta$ such that $x$ has precisely $k$ different $\beta$-expansions. 
For more information on expansions in  non-integer bases we refer to \cite{Baker_2015, Sidorov_2009, Zou_Kong_2015}, and the surveys \cite{deVries-Komornik-2016, Komornik_2011, Sidorov_2003_survey}.

In this paper we consider expansions with digits set $\set{0,1,q}$. 
Given $q>1$, the infinite sequence $(d_i)$ is called a  \emph{$q$-expansion} of $x$, if 
\[
x=((d_i))_q:=\sum_{i=1}^\f\frac{d_i}{q^i},\quad d_i\in\set{0,1,q}\textrm{ for all } i\ge 1.
\]
We emphasize  that   the \emph{digits set} $\set{0,1,q}$  also depends on the base $q$.

For $q>1$ let $E_q$ be the set of points which have a $q$-expansion. Then  $E_q$ is  the attractor of the \emph{iterated function system}  (IFS)
\[
\phi_d(x)=\frac{x+d}{q},\quad d\in\set{0,1,q}.
\]
So, $E_q$ is the non-empty compact set satisfying $E_q=\bigcup_{d\in\set{0,1,q}}\phi_d(E_q)$ (cf.~\cite{Falconer_1990}).  Observe that $\phi_0(E_q)\cap \phi_1(E_q)\ne\emptyset$ for any $q>1$. Then 
  $E_q$ is a  \emph{self-similar set with overlaps}.     Ngai and Wang \cite{Ngai_Wang_2001} gave    the Hausdorff dimension of $E_q$:
\begin{equation}\label{eq:11}
\dim_H E_q=\frac{\log q^*}{\log q}\quad\textrm{for any}\quad q>q^*,
\end{equation}
where $q^*=(3+\sqrt{5})/2$.  Yao and Li \cite{Yao_Li_2015} considered all possible IFSs generating  the set $E_q$.
Zou et al.~\cite{Zou_Lu_Li_2012} considered the set of points in $E_q$ which have a unique $q$-expansion.  In this paper, we investigate    the set of points in $E_q$ having multiple $q$-expansions.

For $k= 1,2,\cdots, \aleph_0$ or $2^{\aleph_0}$, let 
\[
\B_k:=\set{q\in(1,\f): \exists~x\in E_q\textrm{ with precisely }k\textrm{ different }q\textrm{-expansions}}.
\]
  Accordingly, for $q\in \B_k$ let 
\[
\U_q^{(k)}:=\set{x\in E_q: x~\textrm{has precisely}~k~\textrm{different}~q\textrm{-expansions}}.
\]
For simplicity, we write  $\U_q:=\U_q^{(1)}$ for the set of $x\in E_q$ having a unique $q$-expansion, and denote by   $ \U'_q$  the set of all $q$-expansions corresponding to elements of $\U_q$.  

In this paper we will describe the sizes of the sets $\B_k$ and $\U_q^{(k)}$. Our first result is on 
 the set $\B_k$ for $k= 1,2,\cdots,\aleph_0$ or $2^{\aleph_0}$. Clearly, when $k=1$ we have $\B_1=(1,\f)$, since $0$ always has a unique $q$-expansion for any $q>1$. 
When $k= 2,3,\cdots,\aleph_0$ or $2^{\aleph_0}$ we have the following  
\begin{main}
\label{th:11}
Let  $q_c\approx 2.32472$ be the  appropriate root of 
$x^3-3x^2+2x-1=0.$
Then 
\[\B_{2^{\aleph_0}}=(1,\f),\quad \B_{\aleph_0}=[2,\f),\quad \B_k=(q_c,\f)\quad \textrm{for any}\quad k\ge 2.\]
\end{main}

By Theorem \ref{th:11} it follows that for $q\in[2,q_c]$,  any $x\in E_q$ can only have  a unique $q$-expansion, countably infinitely many $q$-expansions, or a continuum of $q$-expansions.  

When $k=1$, the following theorem for the \emph{univoque set} $\U_q=\U_q^{(1)}$ was proven in \cite{Zou_Lu_Li_2012}.
\begin{theorem}\label{th:12}
\mbox{}

\begin{enumerate}[{\rm(i)}]
\item  If $q\in(1, q_c]$, then $\U_q=\set{0,q/(q-1)}$.
\item  If $q\in(q_c, q^*)$, then $\U_q$ contains a continuum of points.
\item If $q\in[q^*,\f)$, then $\dim_H\U_q=\log q_c/\log q$. 
\end{enumerate}
\end{theorem}

Our second result complements Theorem \ref{th:12}, and shows that there is no difference between the Hausdorff dimensions of   $\U_q^{(k)}$  and $\U_q$.
\begin{main}\label{th:13}\mbox{}

\begin{enumerate}[{\rm(i)}]
\item
 $\dim_H\U_{q}>0$ if and only if $q>q_c$.
 
 \item  For any integer  $k\ge 2$ and any $q\in\B_{k}$ we have 
$
\dim_H\U_q^{(k)}=\dim_H\U_q.
$
 \end{enumerate}
\end{main}

As a result of Theorem \ref{th:13} it follows that $q_c$ is indeed the \emph{critical base},  in the sense that $\U_q^{(k)}$ has positive Hausdorff dimension if $q>q_c$, while $\U_q^{(k)}$ has zero Hausdorff dimension if $q\le q_c$. In fact, by Theorems \ref{th:11} and \ref{th:12} (i) it follows that for $q\le q_c$ the set $\U_q=\set{0, q/(q-1)}$ and $\U_q^{(k)}=\emptyset$   for any integer $k\ge 2$.

Our final result focuses on the sizes of  $\U_q^{(\aleph_0)}$ and $\U_q^{(2^{\aleph_0})}$.
\begin{main}\label{th:14}\mbox{}

\begin{enumerate}[{\rm(i)}]
\item
 Let  $q\in\B_{\aleph_0}\setminus(q_c, q^*)$. Then $\U_q^{(\aleph_0)}$ is countably infinite.

\item For any  $q>1$ we have
$
\dim_H\U_q^{(2^{\aleph_0})}=\dim_H E_q.
$
\end{enumerate}
 \end{main}
 
\begin{remark}\label{re:16}
In Lemma \ref{lem:55} we prove a stronger result of Theorem \ref{th:14} (ii), and show  that  the Hausdorff measures of $\U_q^{(2^{\aleph_0})}$ and $E_q$ are the same for any $q>1$, i.e., 
\[
\mathcal{H}^s(\U_q^{(2^{\aleph_0})})=\mathcal{H}^s(E_q)\in(0, \f),
\]
where $s=\dim_H E_q$.
\end{remark}

The rest of the paper is arranged as follows. In Section \ref{sec:unique expansions} we recall some properties of unique $q$-expansions. The proof of Theorem \ref{th:11} for the sets $\B_k$ will be presented in Section \ref{sec:proof of th 11}, and the proofs of Theorems \ref{th:13} and \ref{th:14} for the sets $\U_q^{(k)}$ will be given in Sections \ref{sec:proof of th 13} and \ref{sec:proof of th 14}, respectively. Finally, in Section \ref{sec:remarks} we give some examples and end the paper with some questions.

\section{unique expansions}\label{sec:unique expansions}
In this section we recall some properties of the univoque set $\U_q$ from \cite{Zou_Lu_Li_2012}. 
Recall that  
\begin{equation}\label{eq:qc-q*}
q_c\approx 2.32472\quad\textrm{and}\quad q^*=\frac{3+\sqrt{5}}{2}\approx 2.61803,
\end{equation}
where $q_c$ is the appropriate root of the equation $x^3-3 x^2+2x-1=0$. Note that   for $q\in(1,q^*]$ the attractor $E_q=[0, q/(q-1)]$ is an interval. However, for $q>q^*$ the attractor $E_q$ is a {Cantor set} which contains neither interior nor isolated points. 

Given $q>1$, let $\set{0,1,q}^\mathbb N$ be the set of all infinite sequences $(d_i)$ over the alphabet $\set{0,1,q}$. By a word $\mathbf c$ we mean a finite string of digits $\mathbf c=c_1\ldots c_n$ with each digit $c_i\in\set{0,1,q}$. For two words $\mathbf c=c_1\ldots c_m$ and $\mathbf d=d_1\ldots d_n$, we denote by $\mathbf c\mathbf d=c_1\ldots c_md_1\ldots d_n$ their concatenation. For a positive integer $k$ we write $\mathbf c^k=\mathbf c\cdots \mathbf c$ for the $k$-fold concatenation of $\mathbf c$ with itself. Furthermore, we write $\mathbf c^\f=\mathbf c\mathbf c\cdots$ the infinite periodic sequence with periodic block $\mathbf c$. Throughout the paper we will use lexicographical ordering $\prec, \lle, \succ$ and $\lge$ between sequences. More precisely, for two sequences $(c_i), (d_i)\in\set{0,1,q}^{\mathbb N}$ we say $(c_i)\prec (d_i)$ or $(d_i)\succ (c_i)$ if there exists an integer $n\ge 1$ such that $c_1\ldots c_{n-1}=d_1\ldots d_{n-1}$ and $c_n<d_n$. Furthermore, we say $(c_i)\lle (d_i)$ if $(c_i)\prec (d_i)$ or $(c_i)=(d_i)$. 

Recall that $\U_q$ is the  set of points in $E_q$ with a unique $q$-expansion, and $\U_q'$ is the set of corresponding $q$-expansions. Then  
\[\U_q'=\set{(d_i)\in\set{0,1,q}^\N:((d_i))_q\in\U_q}.\]
The following lexicographical characterization of  $\U'_q$ for $q>q^*$ was established in \cite[Lemma 3.1]{Zou_Lu_Li_2012}. 
\begin{lemma}\label{lem:21}
Let $q>q^*$. Then $(d_i)\in\U_q'$ if and only if 
\[
\left\{
\begin{array}{lll}
(d_{n+i})\prec q 0^\f&\quad\textrm{if}& d_n=0,\\
(d_{n+i})\succ1^\f &\quad\textrm{if}& d_n=1.\\
\end{array}
\right.
\]
\end{lemma}

To describe $\U_q'$ for $q\in(1, q^*]$
we need the following notation.  Let
\[\al(q)=(\al_i(q))\]
 be the \emph{quasi-greedy} $q$-expansion of $q-1$, i.e., the lexicographically largest $q$-expansion of $q-1$ with infinitely many non-zero digits.  We emphasize that $\al(q)$ is well-defined for $q\in(1, q^*]$. By (\ref{eq:qc-q*}) and a direct calculation one can verify that 
 \begin{equation}\label{eq:33}
\al(q_c)=q_c1^\f,\quad\al(q^*)=(q^*)^\f.
\end{equation}

  Note by   
 Theorem \ref{th:12} that for $q\in(1, q_c]$ we have $\U_q=\set{0, q/(q-1)}$, and then $\U_q'=\set{0^\f, q^\f}$. So, it suffices to consider    $\U'_q$  for $q\in(q_c, q^*]$. The following lemma was obtained in \cite[Lemmas 3.1 and 3.2]{Zou_Lu_Li_2012}.
\begin{lemma}\label{lem:22}
 Let  $q\in(q_c,q^*]$. Then 
 \[
 A_q\subseteq \U_q'\subseteq B_q,
 \]
 where $A_q$ is the set of sequences $(d_i)\in\set{0,1,q}^\N$ satisfying
\begin{equation}\label{eq:21}
\left\{
\begin{array}{lll}
(d_{n+i})\prec 1\al(q)&\quad\textrm{if}& d_n=0,\\
1^\f\prec(d_{n+i})\prec\al(q)&\quad\textrm{if}& d_n=1,\\
(d_{n+i})\succ 0q^\f&\quad\textrm{if}& d_n=q,\\
\end{array}
\right.
\end{equation}
and $B_q$ is the set of sequences $(d_i)\in\set{0,1,q}^\N$ satisfying the first two inequalities in (\ref{eq:21}).
\end{lemma}

For $q>1$ let $\Phi:\set{0,1,q}^\N\ra\set{0,1,2}^\N$ be defined by
\[
\Phi((d_i))=(d_i'),
\] 
where $d_i'=d_i$ if $d_i\in\set{0,1}$, and $d_i'=2$ if $d_i=q$.  Clearly, $\Phi$ is bijective and strictly increasing. The following lemma was given in  \cite[Lemma 3.2]{Zou_Lu_Li_2012}.

\begin{lemma}\label{lem:23}
The map $q\ra\Phi(\al(q))$ is strictly increasing in $(1,q^*]$.
\end{lemma}
By (\ref{eq:33}) and Lemma \ref{lem:23} it follows that for any $q\in(q_c, q^*)$ we have $q1^\f\prec \al(q)\prec q^\f$.

\section{Proof of Theorem \ref{th:11}}\label{sec:proof of th 11}
In this section we will investigate the set $\B_k$ of bases $q>1$ in which there exists $x\in E_q$ having   $k$ different $q$-expansions. Excluding the trivial case for $k=1$ that $\B_1=(1,\f)$  we consider $\B_k$ for $k= 2,3,\cdots,\aleph_0$ or $2^{\aleph_0}$.

The following lemma was established in \cite[Theorem 4.1]{Zou_Lu_Li_2012} and \cite[Theorem 1.1]{Ge_Tan_2015_4}.
\begin{lemma}\label{lem:31}
Let $q\in(1,2)$. 
\begin{enumerate}[{\rm(i)}]
\item If $q\in(1,2)$, then any $x\in E_q$  has either a unique $q$-expansion, or a continuum of $q$-expansions. 

\item If $q=2$, then  any $x\in E_q$ can only have  a unique $q$-expansion, countably infinitely many $q$-expansions, or a continuum of $q$-expansions. 
\end{enumerate}
\end{lemma}

  For $q>1$ we recall that $\phi_d(x)=(x+d)/q$ for $d\in\set{0,1,q}$. Let 
  \begin{equation}\label{eq:31}
  S_q:=\left(\phi_0(E_q)\cap\phi_1(E_q)\right)\cup\left(\phi_1(E_q)\cap\phi_q(E_q)\right).
  \end{equation}
Then $S_q$ is associated with the \emph{switch region}, since  any $x\in S_q$ has at least two $q$-expansions. 
More precisely,  any $x\in\phi_0(E_q)\cap\phi_1(E_q)$ has at least two $q$-expansions: one begins with the digit $0$ and one begins with the digit $1$. Accordingly, any $x\in\phi_1(E_q)\cap \phi_q(E_q)$ also has at least two $q$-expansions: one starts with the digit $1$ and one starts with the digit $q$. We point out that the union in (\ref{eq:31}) is disjoint if $q>2$. In particular,  for $q>q^*$ the intersection $\phi_1(E_q)\cap\phi_q(E_q)=\emptyset$.

  For $x\in E_q$ let $\Sigma(x)$ be the set of all $q$-expansions of $x$, i.e., 
  \[
  \Sigma(x):=\set{(d_i)\in\set{0,1,q}^\N: ((d_i))_q=x},
  \]
 and denote its cardinality by $|\Sigma(x)|$. 
  
  We recall from \cite{Baker_2015} that 
  a point $x\in S_q$ is called a \emph{$q$-null infinite  point} if $x$ has an expansion $(d_i)\in\set{0,1,q}^\N$ such that whenever 
  \[x_n:=(d_{n+1}d_{n+2}\cdots)_q\in S_q,\]
   one of the following quantities is infinity, and the other two are finite:
  \[
\left  |\Sigma(\phi_0^{-1}(x_n))\right|,\quad \left  |\Sigma(\phi^{-1}_{1}(x_n))\right|\quad\textrm{and}\quad \left  |\Sigma(\phi^{-1}_{q}(x_n))\right|.
  \]
Then any $q$-null infinite point has countably infinitely many $q$-expansions. 

First we consider the set $\B_{\aleph_0}$, which is based on  the following characterization (cf.~\cite{Baker_2015, Zou_Kong_2015}).
\begin{lemma}\label{lem:32}
$q\in\B_{\aleph_0}$ if and only if $S_q$ contains a $q$-null infinite point. 
\end{lemma}
 
\begin{lemma}\label{lem:33}
$\B_{\aleph_0}=[2,\f)$.
\end{lemma}
\begin{proof}
By Lemma \ref{lem:31} we have  $\B_{\aleph_0}\subseteq[2,\f)$  and $2\in\B_{\aleph_0}$. So, it suffices to prove $(2,\f)\subseteq\B_{\aleph_0}$.

Take $q\in(2,\f)$. Note that $0=(0^\f)_q$ and $q/(q-1)\in(q^\f)_q$  belong to $\U_q$. We claim that 
\[
x=(0q^\f)_q
\]
is a $q$-null infinite point.
Note that $(10^\f)_q=(0q0^\f)_q$. Then 
by the words substitution $10\sim 0q$ it follows that all expansions $1^k0 q^\f, k\ge 0,$ are  $q$-expansions of $x$, i.e.,
\[
\bigcup_{k=0}^\f\set{1^k0q^\f}\subseteq\Sigma(x).
\] 
This implies that $|\Sigma(x)|=\f$.
Furthermore, since $q>2$, the union in (\ref{eq:31}) is disjoint. This implies 
\[
x=(0q^\f)_q=(10q^\f)_q\in\phi_0(E_q)\cap\phi_1(E_q)\setminus\phi_q(E_q).
\]
Then   $\phi_0^{-1}(x)=(q^\f)_q\in\U_q$, $\phi_1^{-1}(x)=x$ and $\phi_q^{-1}(x)\notin E_q$, i.e., 
\[
|\Sigma(\phi_0^{-1}(x))|=1,\quad |\Sigma(\phi^{-1}_1(x))|=\f,\quad |\Sigma(\phi_q^{-1}(x))|=0.
\]

By iteration it follows  that  $x$ is a $q$-null infinite point. Hence,  by Lemma  \ref{lem:32} we have $q\in\B_{\aleph_0}$, and therefore $(2,\f)\subseteq\B_{\aleph_0}$.
\end{proof}

Now we turn to describe the set $\B_k$. By Lemma \ref{lem:31} it follows that $\B_k\subseteq(2,\f)$ for any $k\ge 2$. 
First we consider $\B_2$ and need the following
\begin{lemma}\label{lem:34}
Let $q>2$. Then $q\in\B_2$ if and only if  either
\[
(0(a_i))_q=(1(b_i))_q\qquad\textrm{for some}\quad (a_i), (b_i)\in\U_q',
\]
or  
\[
(1(c_i))_q=(q(d_i))_q\qquad\textrm{for some}\quad(c_i), (d_i)\in\U_q'.
\]
\end{lemma}
\begin{proof}
First we prove the necessary condition.  Take $q\in \B_2$.  Suppose   $x\in E_q$ has  two different $q$-expansions, say
\[
((a_i))_q=x=((b_i))_q.
\]
Then there exists a least integer  $k\ge 1$ such that $a_k\ne b_k$.   Then  
\begin{equation}\label{eq:32}
(a_k a_{k+1}\cdots)_q=(b_kb_{k+1}\cdots)_q\in S_q \quad\textrm{and}\quad  (a_{k+i}), (b_{k+i})\in\U'_q.
\end{equation}
Since $q>2$, it gives that  the union in (\ref{eq:31}) is disjoint. Then the necessity follows by (\ref{eq:32}).

To prove the sufficiency, without loss of generality, we assume  $(0(a_i))_q=(1(b_i))_q$ with $(a_i), (b_i)\in\U_q'$. Note by $q>2$ that the union in (\ref{eq:31}) is disjoint. Then 
\[
(0(a_i))_q=(1(b_i))\in\phi_0(E_q)\cap\phi_1(E_q)\setminus\phi_q(E_q).
\]
This implies that $x$ has exactly  two different $q$-expansions. So, $q\in \B_2$. 
\end{proof}

Recall from (\ref{eq:33}) that $q_c\approx 2.32472$ and $q^*=(3+\sqrt{5})/2$ admit the quasi-greedy expansions $\al(q_c)=q_c1^\f$ and $\al(q^*)=(q^*)^\f.$
In the following lemma we describe the set $\B_2$.
\begin{lemma}\label{lem:35}
 $\B_2=(q_c,\f)$.
\end{lemma}

\begin{proof}
First we show that $\B_2\subseteq(q_c,\f)$. By Lemma \ref{lem:31} it suffices to prove that  any $q\in(2,q_c]$ is not contained in $\B_2$. Take $q\in(2,q_c]$. By Theorem \ref{th:12} we have  $\U'_q=\set{(0^\f), (q^\f)}$.
Then by Lemma \ref{lem:34} it follows that if $q\in\B_2\cap(2,q_c]$ then $q$ must satisfy one of  the following equations
 \[
 (0q^\f)_q=(10^\f)_q \quad\textrm{or}\quad (1q^\f)_q=(q0^\f)_q.
 \]
 This is impossible since  neither equation has a solution in $(2,q_c]$. Hence, $\B_2\subseteq(q_c, \f)$.

Now we turn to prove $(q_c,\f)\subseteq\B_2$.
By Lemmas \ref{lem:21} and \ref{lem:34},  one can verify that for any $q>q^*$ the number 
\[
x=(0q0^\f)_q=(10^\f)_q
\]
has precisely  two different $q$-expansions. This implies that $(q^*,\f)\subseteq\B_2$.

For $q\in(q_c, q^*]$, one has by (\ref{eq:33}) that $\al(q_c)=q_c1^\f$ and $\al(q^*)=(q^*)^\f$. Then by Lemma \ref{lem:23}  there exists an integer $m\ge 0$  such that 
 \[
\al(q)\succ q1^mq0^\f.
 \]
 Hence, by Lemmas \ref{lem:22} and \ref{lem:34} one can verify that 
 \[
 y=(0q (1^{m+1}q)^\f)_q=(10(1^{m+1}q)^\f)_q
 \]
 has  precisely two different $q$-expansions. So, $(q_c,q^*]\subseteq\B_2$, and   the proof is complete.
\end{proof}

\begin{lemma}\label{lem:36}
$\B_k=(q_c,\f)$ for any $k\ge 3$.
\end{lemma}
\begin{proof}
First we prove  $\B_{k}\subseteq\B_2$ for any $k\ge 3$.  By Lemma \ref{lem:31} it follows that $\B_k\subseteq(2,\f)$. 
Take $q\in\B_k$ with $k\ge 3$.  Suppose $x\in E_q$ has exactly $k$ different $q$-expansions. Since $q>2$,  the union in (\ref{eq:31}) is disjoint. This implies that there exists a word $d_1\cdots d_n$ such that 
 \[\phi_{d_1}^{-1}\circ\cdots\circ\phi_{d_n}^{-1}(x)\]
  has exactly two different  $q$-expansions. So, $q\in\B_2$. Hence, $\B_k\subseteq \B_2$ for any $k\ge 3$.

Now we  prove $\B_2\subseteq\B_k$ for any $k\ge 3$.
Note by  Lemma \ref{lem:35} that $\B_2=(q_c, \f)$. Then it suffices  to prove 
$ (q_c,\f)\subseteq\B_k.$
First we prove $(q^*,\f)\subseteq\B_k$. Take $q\in (q^*, \f)$. We claim  that  for any $k\ge 1$,
\[
x_k=(0q^{k-1}(1q)^\f)_q
\]
has  precisely $k$ different $q$-expansions. We will prove this by induction on $k$. 

For $k=1$ one can easily check by using Lemma \ref{lem:21} that $x_1=(0(1q)^\f)_q\in\U_q$. Suppose $x_k$ has exactly $k$ different $q$-expansions. Now we consider $x_{k+1}$, which can be written as
\[
x_{k+1}=(0q^k(1q)^\f)_q=(10q^{k-1}(1q)^\f)_q.
\]
By Lemma \ref{lem:21} we have $q^k(1q)^\f\in\U_{q}'$. Moreover, by the induction hypothesis $(0q^{k-1}(1q)^\f)_q=x_k$ has exactly $k$ different $q$-expansions. Then $x_{k+1}$ has at least  $k+1$ different $q$-expansions. 
On the other hand, since $q>q^*>2$,    the union in (\ref{eq:31}) is disjoint. Then 
\[x_{k+1}\in\phi_0(E_q)\cap\phi_1(E_q)\setminus\phi_q(E_q).
\]
This implies that $x_{k+1}$ indeed has  $k+1$ different $q$-expansions. By induction this proves the claim, and hence $(q^*, \f)\subseteq\B_k$ for all $k\ge 3$.

It remains to prove $(q_c,q^*]\subseteq\B_k$.  Take $q\in(q_c,q^*]$. By (\ref{eq:33}) and Lemma \ref{lem:23} there exists an integer $m\ge 0$ such that 
\begin{equation}\label{eq:34}
\al(q)\succ q1^mq0^\f.
\end{equation}
We claim that 
\[
y_k=(0q^{k-1}(1^{m+1}q)^\f)_q
\]
has  exactly $k$ different $q$-expansions. Again, this will be proven by induction on $k$.

If $k=1$, then by using (\ref{eq:34}) in Lemma \ref{lem:22} it gives that $y_1=(0(1^{m+1}q)^\f)_q$ has a unique $q$-expansion. Suppose $y_k$ has exactly $k$ different $q$-expansions. Now we consider  
\[
y_{k+1}=(0q^k(1^{m+1}q)^\f)_q=(10q^{k-1}(1^{m+1}q)^\f)_q.
\]
By (\ref{eq:34}) and Lemma \ref{lem:22} it yields that $q^{k}(1^{m+1}q)^\f\in\U_{q}'$. Furthermore, by the induction hypothesis   $(0q^{k-1}(1^{m+1}q)^\f)_q=y_k$ has exactly $k$ different $q$-expansions. This implies that $y_{k+1}$ has at least $k+1$ different $q$-expansions.
 On the other hand, note that $q>q_c>2$, and therefore the union in (\ref{eq:31}) is disjoint. So, 
 $
 y_{k+1}\in\phi_0(E_q)\cap\phi_1(E_q)\setminus\phi_q(E_q),
 $
which implies that    $y_{k+1}$  indeed has  $k+1$ different $q$-expansions. By induction this proves the claim, and then $(q_c, q^*]\subseteq\B_k$ for all $k\ge 3$. This completes the proof.
\end{proof}

\begin{proof}[Proof of Theorem \ref{th:11}]
By Lemmas \ref{lem:33}, \ref{lem:35} and \ref{lem:36} it suffices to prove $\B_{2^{\aleph_0}}=(1,\f)$. This can be verified by observing that 
\[
x=((100)^\f)_q\in\U_q^{(2^{\aleph_0})}
\]
for any $q>1$, because by the word substitution $10\sim 0q$ one can show that $x$ indeed has a continuum of different $q$-expansions. 
\end{proof}

\section{Proof of Theorem \ref{th:13}}\label{sec:proof of th 13}
For $q>1$ and $k\in\N$ we recall that $\U_q^{(k)}$ is the set of $x\in[0, q/(q-1)]$ having precisely $k$ different $q$-expansions. In this section we are going to investigate  the Hausdorff dimension of $\U_q^{(k)}$. First we show that $q_c\approx 2.32472$ is  the critical base for $\U_q$.

\begin{lemma}\label{lem:41}
Let $q>1$. Then $\dim_H\U_q>0$ if and only if $q>q_c$.
\end{lemma}
\begin{proof}
 The necessity follows from Theorem \ref{th:12} (i).  For the sufficiency we take $q\in(q_c,\f)$.  If $q>q^*$, then by Theorem \ref{th:12} (iii) we have
 \[
 \dim_H\U_q=\frac{\log q_c}{\log q}>0.
 \]
So it remains to prove $\dim_H\U_q>0$ for any $q\in(q_c, q^*]$.
 
Take $q\in(q_c, q^*]$.  Recall from (\ref{eq:33}) that $\al(q_c)=q_c1^\f$ and $\al(q^*)=(q^*)^\f$. Then by Lemma \ref{lem:23}  there exists an integer $m\ge 0$ such that 
 $
 \al(q)\succ q 1^mq0^\f.
 $
Whence, by Lemma \ref{lem:22} one can verify that all sequences in
\[
\Delta'_m:=\prod_{i=1}^\f\set{q1^{m+1}, 1^{m+2}}
\]
excluding those ending with $1^\f$  belong to $\U_q'$.  
This  implies that 
\begin{equation}\label{eq:41}
\dim_H\U_q\ge\dim_H\Delta_m(q),
\end{equation}
where $\Delta_m(q):=\set{((d_i))_q: (d_i)\in\Delta_m'}$.
Note that $\Delta_m(q)$ is a self-similar set generated by the IFS
\[
f_1(x)=\frac{x}{q^{m+2}}+(q1^{m+1}0^\f)_q,\quad f_2(x)=\frac{x}{q^{m+2}}+(1^{m+2}0^\f)_q,
\]
which satisfies the open set condition (cf.~\cite{Falconer_1990}). Therefore, by (\ref{eq:41}) we conclude that
\[
\dim_H\U_q\ge\dim_H\Delta_m(q)=\frac{\log 2}{(m+2)\log q}>0.\qedhere
\]
\end{proof}

In the following we will consider the Hausdorff dimension of $\U_q^{(k)}$ for any $k\ge 2$, and prove $\dim_H\U_q^{(k)}=\dim_H\U_q$. The upper bound   of $\dim_H\U_q^{(k)}$ is easy.
\begin{lemma}\label{lem:42}
Let $q>1$. Then $\dim_H\U_q^{(k)}\le\dim_H\U_q$ for any $k\ge 2$.
\end{lemma}
\begin{proof}
Recall that $\phi_d(x)=(x+d)/q$ for $d\in\set{0,1,q}$. Then   the lemma follows  by observing  that for any $k\ge 2$,
\[
\U_q^{(k)}\subseteq\bigcup_{n=1}^\f\bigcup_{d_1\cdots d_n\in\set{0,1,q}^n}\phi_{d_1}\circ\cdots\circ\phi_{d_n}(\U_q),\]
and the countable stability of Hausdorff dimension.
\end{proof}

\medskip

For the lower bound of $\dim_H\U_q^{(k)}$ we need more.  By Lemmas \ref{lem:41} and \ref{lem:42} it follows that 
\[\dim_H\U_q^{(k)}=0=\dim_H\U_q\quad\textrm{for any }q\le q_c.\]
 So,   it suffices to consider $q>q_c$.
Let 
\[
F_q'(1):=\set{(d_i)\in\U_q': d_1=1}
\]
be the \emph{follower set} in $\U_q'$ generated by the word $1$,
and let $F_q(1)$ be the set of $x\in E_q$ which have a $q$-expansion in $F_q'(1)$, i.e., 
$F_q(1)=\set{((d_i))_q: (d_i)\in F_q'(1)}.$

\begin{lemma}\label{lem:43}
Let $q>q_c$. Then $\dim_H\U_q^{(k)}\ge \dim_H F_q(1)$ for any $k\ge 1$.
\end{lemma}
\begin{proof}
  For $k\ge 1$ and $q>q_c$ let 
\[
\La^k_q:=\set{((d_i))_q: d_1\cdots d_k=0q^{k-1}, (d_{k+i})\in F'_q(1)}.
\]
Then $\La^k_q=\phi_0\circ\phi_q^{k-1}(F_q(1))$, and therefore 
$\dim_H\La^k_q=\dim_H F_q(1).$
So it suffices to prove $\La^k_q\subseteq\U_q^{(k)}$. 
Arbitrarily take
 \[
x_k=(0q^{k-1}(c_i))_q\in\La^k_q\quad\textrm{ with}\quad (c_i)\in F'_q(1).
\] 
We will prove by induction on $k$ that $x_k$   has exactly  $k$ different $q$-expansions. 

For $k=1$,  by Lemmas \ref{lem:21} and $\ref{lem:22}$ it follows that 
$
x_1=(0(c_i))_q\in\U_q.
$
Suppose $x_k=(0q^{k-1}(c_i))_q$ has precisely  $k$ different $q$-expansions. Now we consider $x_{k+1}$, which can be expanded as
\[
x_{k+1}=(0q^k(c_i))_q=(10q^{k-1}(c_i))_q.
\]
By Lemmas \ref{lem:21} and \ref{lem:22} we have $q^k(c_i)\in\U'_q$, and by the induction hypothesis it yields that $(0q^{k-1}(c_i))_q=x_k$ has $k$ different $q$-expansions. This implies that $x_{k+1}$ has at least $k+1$
 different $q$-expansions. 
 On the other hand, since $q>q_c>2$, it gives that the union in (\ref{eq:31}) is disjoint. So,
 $x_{k+1}\in\phi_0(E_q)\cap\phi_1(E_q)\setminus\phi_q(E_q),$
  which  implies that  $x_{k+1}$ indeed has $k+1$ different $q$-expansions.
  
   By induction this proves $x_k\in\U_q^{(k)}$ for all $k\ge 1$. Since $x_k$ was taken arbitrarily  from $\La_q^k$,   we conclude that $\La_q^k\subseteq\U_q^{(k)}$ for any $k\ge 1$. The proof is complete.
\end{proof}
\begin{lemma}\label{lem:44}
Let $q>q_c$. Then $\dim_H F_q(1)\ge \dim_H\U_q$.
\end{lemma}
\begin{proof}
First we consider  $q>q^*$. By Lemma \ref{lem:21} one can show that $\U'_q$ is contained in  an irreducible sub-shift of finite type $X'_A$ over the states $\set{0,1,q}$ with adjacency matrix 
\begin{equation}\label{eq:42}
A=\left(
\begin{array}{ccc}
1&1&0\\
0&1&1\\
1&1&1
\end{array}
\right).
\end{equation}
Moreover, the complement set $X'_A\setminus\U_q'$ contains all sequences ending with $1^\f$. This implies that 
\begin{equation}\label{eq:43}
\dim_H\U_q=\dim_H X_A(q),
\end{equation}
where $X_A(q):=\set{((d_i))_q: (d_i)\in X_A'}$. Note that $X_A(q)$ is a graph-directed set satisfying the open set condition (cf.~\cite[Theorem 3.4]{Zou_Lu_Li_2012}), and the sub-shift of  finite type   $X_A'$ is irreducible. Then by (\ref{eq:43}) it follows that 
\[
\dim_H\U_q=\dim_H X_A(q)=\dim_H F_q(1).
\]

Now we consider $q\in(q_c,q^*]$. By  Lemma \ref{lem:22} it follows that  
\[
\begin{split}
\U_q'\subseteq\set{q^\f}\cup \bigcup_{k=0}^\f\set{q^k 0^\f} \cup\bigcup_{k=0}^\f\bigcup_{m=0}^\f\set{q^k 0^m F_q'(1)},
\end{split}
\]
where 
\[q^k 0^m F_q'(1):=\set{(d_i): d_1\cdots d_{k+m}=q^k0^m,  (d_{k+m+i})\in F_q'(1)}.\]
This implies that 
$
\dim_H\U_q\le \dim_H F_q(1).
$
\end{proof}

\begin{proof}[Proof of Theorem \ref{th:13}]
The theorem follows directly by Lemmas \ref{lem:41}--\ref{lem:44}.
\end{proof}

\section{Proof of Theorem \ref{th:14} }\label{sec:proof of th 14}

In this section we will consider the set  $\U_q^{(\aleph_0)}$ which consists of all $x\in E_q$ having countably infinitely many $q$-expansions. 
\begin{lemma}
\label{lem:51}
For any $q\in\B_{\aleph_0}$ the set $\U_q^{(\aleph_0)}$ contains infinitely many points.
\end{lemma}
\begin{proof}
Let $q\in\B_{\aleph_0}$. By Theorem \ref{th:11} we have $q\in[2,\f)$. Then it suffices to show that for any $k\ge 1$,
\[
z_k:=(0^kq^\f)_q
\]
is a $q$-null infinite points, and thus  $z_k\in\U_q^{(\aleph_0)}$.

If $q>2$, then by the proof
of Lemma \ref{lem:33} it yields that $z_1=(0q^\f)_q$ is a $q$-null infinite point.
Moreover, note that 
$
z_k=\phi_0^{k-1}(z_1)\notin S_q
$
for any $k\ge 2$.
This implies that all of these points $z_k, k\ge 1$, are $q$-null infinite points. So,   
$
\set{ z_k: k\ge 1}\subseteq\U_q^{(\aleph_0)}.
$

If $q=2$, then by using the substitutions
\[
0q\sim 10,\quad 0q^\f=1^\f=q0^\f,
\]
one can also show that $z_k$ is a $q$-null infinite point.  In fact,
  all of the $q$-expansions of $z_k=(0^kq^\f)_q$ are of the form
\[
0^kq^\f,\quad 0^{k-1}1^\f,\quad 0^{k-1}1^m0q^\f\quad\textrm{and}\quad 0^{k-1}1^{m-1}q0^\f, 
\]
where $m\ge 1$.
Therefore, $z_k\in\U_q^{(\aleph_0)}$ for any $k\ge 1$. 
\end{proof}

By Lemma \ref{lem:51} it follows that $\U_q^{(\aleph_0)}$ is at least countably infinite for any $q\in\B_{\aleph_0}=[2,\f)$. In the following lemma we show that $\U_q^{(\aleph_0)}$ is indeed countably infinite if $q\ge q^*$.
 
\begin{lemma}\label{lem:52}
Let $q\ge q^*$. Then  $\U_q^{(\aleph_0)}$ is at most countable.
\end{lemma}
\begin{proof}
Let $x\in\U_q^{(\aleph_0)}$. Then $x$ has a $q$-expansion $(d_i)$ such that 
\[
\left|\Sigma(x_n)\right|=\f\quad\textrm{for infinitely many }n\in\N,
\]
where $x_n:=((d_{n+i}))_q$. This implies that $(d_i)$ can not end in $\U_q'$.

Note by the proof of Lemma \ref{lem:44} that $\U_q'\subseteq X_A'$, where $X_A'$ is a sub-shift of finite type over the state $\set{0,1,q}$ with adjacency matrix $A$
defined in (\ref{eq:42}). Moreover, $X_A'\setminus \U_q'$ is at most countable (cf.~\cite[Theorem 3.4]{Zou_Lu_Li_2012}). Note that the expansion $(d_i)$ of $x\in\U_q^{(\aleph_0)}$ does not end in $\U_q'$. Then it suffices to prove that the sequence $(d_i)$ must end in $X_A'$.

Suppose on the contrary that $(d_i)$ does not end in $X_A'$. Then  by (\ref{eq:42}) the word $0q$ or $10$ occurs infinitely many times in $(d_i)$.  Using the word substitution $0q\sim 10$ this implies that $x=((d_i))_q$ has a continuum of $q$-expansions, leading to a
contradiction with $x\in\U_q^{(\aleph_0)}$. 
\end{proof}
Furthermore,  we can prove that  $\U_q^{(\aleph_0)}$ is also countably infinite  for $q\in[2,q_c]$.
\begin{lemma}\label{lem:53}
Let $q\in[2,q_c]$. Then  $\U_q^{(\aleph_0)}$ is at most countable.
\end{lemma}
\begin{proof}
 Take $q\in[2, q_c]$. By Theorems \ref{th:11}  and \ref{th:12} it follows that any $x\in E_q$ with  $|\Sigma(x)|<\f$ must belong to $\U_q=\set{0,q/(q-1)}$.  Suppose $x\in\U_q^{(\aleph_0)}$. Then there exists a word $d_1\cdots d_n$ such that 
 \[
 \phi_{d_1}^{-1}\circ\cdots\circ\phi_{d_n}^{-1}(x)\in\U_q.
 \]
This implies that the set $\U_q^{(\aleph_0)}$ is at most countable, since
 \[
 \U_q^{(\aleph_0)}\subseteq\bigcup_{n=1}^\f\bigcup_{d_1\cdots d_n\in\set{0,1,q}^n}\phi_{d_1}\circ\cdots\circ\phi_{d_n}\left(\U_q\right).\qedhere
 \]
\end{proof}

When $q\in(q_c, q^*)$, one might expect that $\U_q^{(\aleph_0)}$ is also countably infinite. Unfortunately, we are not able to prove this. Instead,  we show that  the Hausdorff dimension of   $\U_q^{(\aleph_0)}$ is strictly smaller than $\dim_H E_q=1$.   
 \begin{lemma}\label{lem:54}
 For $q\in(q_c, q^*)$ we have  $ \dim_H\U_q^{(\aleph_0)}\le\dim_H\U_q<1$.
 \end{lemma}
 \begin{proof}
 Take $q\in(q_c, q^*)$. Note that 
 \[
 \U_q^{(\aleph_0)}\subseteq\bigcup_{n=1}^\f\bigcup_{d_1\cdots d_n\in\set{0,1,q}^n}\phi_{d_1}\circ\cdots\circ\phi_{d_n}(\U_q).
 \]
 By using the countable stability of Hausdorff dimension this implies that $\dim_H\U_q^{(\aleph_0)}\le\dim_H\U_q$. In the following it suffices to prove $\dim_H\U_q<1$.

 Note that $\U_q'\subseteq X_A'$, where $X_A'$ is the sub-shift of finite type over the state $\set{0,1, q}$ with adjacency matrix $A$ defined in (\ref{eq:42}). Then 
 \[
 \U_q\subseteq X_A(q)=\set{((d_i))_q: (d_i)\in X_A'}.
 \]
 Note that $X_A(q)$ is a graph-directed set (cf.~\cite{Mauldin_Williams_1988}). This implies that 
 \[
 \dim_H\U_q\le\dim_H X_A(q)\le \frac{\log q_c}{\log q}<1.\qedhere
 \]
  \end{proof}

At the end of this section  we investigate the set $\U_q^{(2^{\aleph_0})}$ which consists of all points having a continuum of $q$-expansions,  and show that $\U_q^{(2^{\aleph_0})}$ has full Hausdorff measure.
 \begin{lemma}\label{lem:55}
 For any  $q>1$ we have
 \[
 \mathcal{H}^{\dim_H E_q}(\U_q^{(2^{\aleph_0})})=\mathcal{H}^{\dim_H E_q}(E_q)\in(0,\f).
 \]
 \end{lemma}
 \begin{proof}
Clearly, for $q\in(1,q^*]$ we have $E_q=[0, q/(q-1)]$, and then $\mathcal H^{\dim_H E_q} (E_q)\in(0,\f)$.  Moreover, for $q>q^*$ we have by (\ref{eq:11}) that $\dim_H E_q=\log q^*/\log q$, and the set $E_q$ has positive and finite  Hausdorff measure (cf.~\cite{Ngai_Wang_2001}). Therefore,   
 \begin{equation}\label{eq:44}
0< \mathcal{H}^{\dim_H E_q}(E_q)<\f\quad\textrm{for any}\quad q>1.
 \end{equation}
  
 First we prove the lemma for $q\le q^*$.
By Theorems \ref{th:11} and  \ref{th:12} it follows that for any $q\in(1,q^*]$, 
\[
\dim_H\U_q^{(k)}=\dim_H\U_q<1=\dim_H E_q \quad\textrm{for any}\quad k\ge 2.
\]
Moreover, by Lemmas \ref{lem:52}--\ref{lem:54} we have 
$\dim_H\U_q^{(\aleph_0)}<1.$ Observe that 
 \begin{equation}\label{eq:45}
 E_q=\U_q^{(2^{\aleph_0})}\cup\U_q^{(\aleph_0)}\cup\bigcup_{k=1}^\f\U_q^{(k)}\quad\textrm{for any }q>1.
 \end{equation}
Therefore,  by (\ref{eq:44}) and (\ref{eq:45}) we have $\mathcal{H}^{\dim_H E_q}(\U_q^{(2^{\aleph_0})})=\mathcal{H}^{\dim_H E_q}(E_q)\in(0,\f).
$

Now   we consider  $q> q^*$. By Theorems \ref{th:12} (iii), \ref{th:13} and (\ref{eq:11}) it follows that 
\[
\dim_H\U_q^{(k)}=\frac{\log q_c}{\log q}<\frac{\log q^*}{\log q}=\dim_H E_q
\]
for any $k\ge 1$. Moreover, by Lemma \ref{lem:52} we have $\dim_H\U_q^{(\aleph_0)}=0$. Again,   by (\ref{eq:44}) and (\ref{eq:45}) it follows that $\mathcal{H}^{\dim_H E_q}(\U_q^{(2^{\aleph_0})})=\mathcal{H}^{\dim_H E_q}(E_q)\in(0,\f).
$ This completes the proof.
 \end{proof}

 \begin{proof}[Proof of Theorem \ref{th:14}]
 The theorem follows by Lemmas \ref{lem:51}--\ref{lem:53} and \ref{lem:55}.
 \end{proof}

\section{Examples and final remarks}\label{sec:remarks}
In this section we consider some examples. The first example is an application of Theorems \ref{th:11}--\ref{th:14} to expansions with deleted digits set. 
\begin{example}\label{ex:1}
Let $q=3$. We consider $q$-expansions with digits set $\set{0,1,3}$. This is a special case of expansions with deleted digits (cf.~\cite{Pollicott-Simon_1995}). Then
\[
E_3=\set{\sum_{i=1}^\f\frac{d_i}{3^i}: d_i\in\set{0,1,3}}.
\]
By Theorems \ref{th:12} and \ref{th:13} we have 
\[
\dim_H\U_3^{(k)}=\dim_H\U_3=\frac{\log q_c}{\log 3}\approx 0.767877
\]
for any $k\ge 2$. This means that the set $\U_3^{(k)}$ consisting of all points in  $E_3$ with precisely $k$ different triadic expansions has the same Hausdorff dimension $\log q_c/\log 3$ for any integer $k\ge 1$.
 Moreover, by Theorem \ref{th:14} it follows that $\U_3^{(\aleph_0)}$ is countably infinite, and
\[
\dim_H\U_3^{(2^{\aleph_0})}=\dim_H E_3=\frac{\log q^*}{\log 3}\approx 0.876036.
\]
\end{example}
Theorem \ref{th:12} gives a uniform formula for the Hausdorff dimension of $\U_q$ for $q\in[q^*, \f)$. Excluding the trivial case for $q\in(1, q_c]$ that $\U_q=\set{0, q/(q-1)}$, it would be interesting to ask whether the Hausdorff dimension of $\U_q$ can be determined for  $q\in(q_c, q^*)$.  
In the following we give an example for which the Hausdorff dimension of $\U_q$ can be explicitly calculated.
\begin{example}\label{ex:2}
Let $q=1+\sqrt{2}\in(q_c, q^*)$. Then 
\[
(q0^\f)_q=(1qq0^\f)_q\quad \textrm{and}\quad\alpha(q)=(q1)^\f.
\]
Moreover, the quasi-greedy $q$-expansion of $q-1$ with alphabet $\set{0, q-1, q}$ is $q(q-1)^\f$. Therefore, 
by Lemmas 3.1 and 3.2 of \cite{Zou_Lu_Li_2012} it follows that $\U_q'$ is the set of 
sequences $(d_i)\in\set{0,1,q}^\f$ satisfying
\[
\left\{
\begin{array}{lll}
d_{n+1}d_{n+2}\cdots\prec (1q)^\f& \textrm{if}& d_n=0,\\
1^\f<d_{n+1}d_{n+2}\cdots\prec (q1)^\f&\textrm{if}& d_n=1,\\
d_{n+1}d_{n+2}\cdots\succ 01^\f&\textrm{if}& d_n=q.
\end{array}
\right.
\]

Let $X_A'$ be the sub-shift of finite type over the states 
\[\set{00, 01, 11, 1q, q0, q1,qq}\]
with adjacency matrix 
\[
A=\left(\begin{array}{ccccccc}
1&1&0&0&0&0&0\\
0&0&1&1&0&0&0\\
0&0&1&1&0&0&0\\
0&0&0&0&1&1&0\\
0&1&0&0&0&0&0\\
0&0&1&1&0&0&0\\
0&0&0&0&1&1&1
\end{array}
\right).
\]
Then one can verify that $\U_q'\subseteq X_A'$, and  $X_A'\setminus\U_q'$ contains all sequences ending with $1^\f$ or $(1q)^\f$. This implies that
\[
\dim_H\U_q=\dim_H X_A(q),
\]
where $X_A(q)=\set{((d_i))_q: (d_i)\in X_A'}$. 
Note that $X_A(q)$ is a graph-directed set satisfying the open set condition (cf.~\cite{Mauldin_Williams_1988}). Then by Theorem \ref{th:13} we have
\[
\dim_H\U_q^{(k)}=\dim_H\U_q=\frac{h(X_A')}{\log q}\approx 0.691404.
\]

Furthermore, by the word substitution $q00\sim 1qq$ and in a similar way as in the proof of Lemma \ref{lem:52} one can show that $\U_q^{(\aleph_0)}$ is countably infinite.  
Finally, by Theorem \ref{th:14} we have $\dim_H\U_q^{(2^{\aleph_0})}=\dim_H E_q=1$.
\end{example}

\emph{Question 1}. Can we give a uniform formula for the Hausdorff dimension of $\U_q$ for $q\in(q_c, q^*)$? 

In beta expansions  we know that the dimension function of the univoque set has a Devil's staircase behavior (cf.~\cite{Komornik_Kong_Li_2015_1}). 

\emph{Question 2}. Does  the dimension function $D(q):=\dim_H\U_q$ have a Devil's staircase behavior in the interval $(q_c, q^*)$?

By Theorem \ref{th:14} one has that 
$\U_q^{(\aleph_0)}$ is countable for any $q\in\B_2\setminus(q_c, q^*)$. Moreover, in Lemma \ref{lem:54} we show that 
$\dim_H\U_q^{(\aleph_0)}\le\dim_H\U_q<1$
 for any $q\in(q_c, q^*)$. In view of Example \ref{ex:2} we ask the following 
  
 \emph{Question 3}. Does there exist a  $q\in(q_c, q^*)$ such that    $\U_q^{(\aleph_0)}$ has positive Hausdorff dimension?

\section*{Acknowledgements}
The second author was supported by NSFC No.~11701302 and K.C. Wong Magna Fund at Ningbo University. 
The third author  was supported by   NSFC No.~11401516 and Jiangsu Province Natural
Science Foundation for the Youth no BK20130433. The forth author  was
supported by  NSFC No.~11271137, 11671147 and  in part by Science and Technology Commission of Shanghai Municipality (No. 18dz2271000)


\begin{thebibliography}{10}

\bibitem{Baker_2015}
S.~Baker.
\newblock On small bases which admit countably many expansions.
\newblock {\em J. Number Theory}, 147:515--532, 2015.

\bibitem{Dajani_DeVries_2007}
K.~Dajani and M.~de~Vries.
\newblock Invariant densities for random {$\beta$}-expansions.
\newblock {\em J. Eur. Math. Soc. (JEMS)}, 9(1):157--176, 2007.

\bibitem{DeVries_Komornik_2008}
M.~de~Vries and V.~Komornik.
\newblock Unique expansions of real numbers.
\newblock {\em Adv. Math.}, 221(2):390--427, 2009.

\bibitem{deVries-Komornik-2016}
M.~de~Vries and V.~Komornik.
\newblock Expansions in non-integer bases.
\newblock In {\em Combinatorics, words and symbolic dynamics}, volume 159 of
  {\em Encyclopedia Math. Appl.}, pages 18--58. Cambridge Univ. Press,
  Cambridge, 2016.

\bibitem{Erdos_Joo_Komornik_1990}
P.~Erd\H{o}s, I.~Jo\'{o}, and V.~Komornik.
\newblock Characterization of the unique expansions $1=\sum_{i=1}^\infty
  q^{-n_i}$ and related problems.
\newblock {\em Bull. Soc. Math. France}, 118:377--390, 1990.

\bibitem{Erdos_Horvath_Joo_1991}
P.~Erd{\H{o}}s, M.~Horv{\'a}th, and I.~Jo{\'o}.
\newblock On the uniqueness of the expansions {$1=\sum q^{-n_i}$}.
\newblock {\em Acta Math. Hungar.}, 58(3-4):333--342, 1991.

\bibitem{Erdos_Joo_1992}
P.~Erd{\H{o}}s and I.~Jo{\'o}.
\newblock On the number of expansions {$1=\sum q^{-n_i}$}.
\newblock {\em Ann. Univ. Sci. Budapest. E\"otv\"os Sect. Math.}, 35:129--132,
  1992.

\bibitem{Falconer_1990}
K.~Falconer.
\newblock {\em Fractal geometry}.
\newblock John Wiley \& Sons Ltd., Chichester, 1990.
\newblock Mathematical foundations and applications.

\bibitem{Ge_Tan_2015_4}
Y.~Ge and B.~Tan.
\newblock Numbers with countable expansions in base of generalized golden
  ratios.
\newblock {\em arXiv:1504.01704}, 2015.

\bibitem{Glendinning_Sidorov_2001}
P.~Glendinning and N.~Sidorov.
\newblock Unique representations of real numbers in non-integer bases.
\newblock {\em Math. Res. Lett.}, 8:535--543, 2001.

\bibitem{Komornik_2011}
V.~Komornik.
\newblock Expansions in noninteger bases.
\newblock {\em Integers}, 11B:Paper No. A9, 30, 2011.

\bibitem{Komornik_Kong_Li_2015_1}
V.~Komornik, D.~Kong, and W.~Li.
\newblock Hausdorff dimension of univoque sets and devil's staircase.
\newblock {\em Adv. Math.}, 305:165--196, 2017.

\bibitem{Kong_Li_Dekking_2010}
D.~Kong, W.~Li, and F.~M. Dekking.
\newblock Intersections of homogeneous {C}antor sets and beta-expansions.
\newblock {\em Nonlinearity}, 23(11):2815--2834, 2010.

\bibitem{Mauldin_Williams_1988}
R.~D. Mauldin and S.~C. Williams.
\newblock Hausdorff dimension in graph directed constructions.
\newblock {\em Trans. Amer. Math. Soc.}, 309(2):811--829, 1988.

\bibitem{Ngai_Wang_2001}
S.~M. Ngai and Y.~Wang.
\newblock Hausdorff dimension of self-similar sets with overlaps.
\newblock {\em J. Lond. Math. Soc}, 63:655--672, 2001.

\bibitem{Parry_1960}
W.~Parry.
\newblock On the $\beta$-expansions of real numbers.
\newblock {\em Acta Math. Acad. Sci. Hungar.}, 11:401--416, 1960.

\bibitem{Pollicott-Simon_1995}
M.~Pollicott and K.~Simon.
\newblock The hausdorff dimension of $\lambda$-expansions with deleted digits.
\newblock {\em Trans. Amer. Math. Soc.}, 347(3):967--983, 1995.

\bibitem{Renyi_1957}
A.~R\'{e}nyi.
\newblock Representations for real numbers and their ergodic properties.
\newblock {\em Acta Math. Acad. Sci. Hungar.}, 8:477--493, 1957.

\bibitem{Sidorov_2003}
N.~Sidorov.
\newblock Almost every number has a continuum of {$\beta$}-expansions.
\newblock {\em Amer. Math. Monthly}, 110(9):838--842, 2003.

\bibitem{Sidorov_2003_survey}
N.~Sidorov.
\newblock Arithmetic dynamics.
\newblock In {\em Topics in dynamics and ergodic theory}, volume 310 of {\em
  London Math. Soc. Lecture Note Ser.}, pages 145--189. Cambridge Univ. Press,
  Cambridge, 2003.

\bibitem{Sidorov_2009}
N.~Sidorov.
\newblock Expansions in non-integer bases: lower, middle and top orders.
\newblock {\em J. Number Theory}, 129(4):741--754, 2009.

\bibitem{Yao_Li_2015}
Y.~Yao and W.~Li.
\newblock Generating iterated function systems for a class of self-similar sets
  with complete overlap.
\newblock {\em Publ. Math. Debrecen}, 87(1-2), 2015.

\bibitem{Zou_Kong_2015}
Y.~Zou and D.~Kong.
\newblock On a problem of countable expansions.
\newblock {\em Journal of Number Theory}, 158:134--150, 2016.

\bibitem{Zou_Lu_Li_2012}
Y.~Zou, J.~Lu, and W.~Li.
\newblock Unique expansion of points of a class of self-similar sets with
  overlaps.
\newblock {\em Mathematika}, 58(2):371--388, 2012.

\end{thebibliography}

\end{document}